\documentclass[12pt,reqno]{amsart}
\usepackage{amsmath,amsfonts,amssymb,amsthm,amsbsy,latexsym,cite,amscd}
\usepackage[cp1251]{inputenc}
\usepackage[T2A]{fontenc}
\usepackage[english]{babel}
\usepackage[colorlinks,breaklinks,linkcolor=blue,citecolor=blue,menucolor=blue]{hyperref}

\multlinegap=\parindent

\begin{document}

\newtheorem*{etheorem}{Theorem}
\newtheorem{eproposition}{Proposition}[section]
\newtheorem{ecorollary}{Corollary}

\renewcommand{\thefootnote}{}

\title[On~the~subgroup separability of~the~free product of~groups]{On~the~subgroup separability\\ of~the~free product of~groups}

\author{E.~V.~Sokolov}
\address{Ivanovo State University, Russia}
\email{ev-sokolov@yandex.ru}

\begin{abstract}
Suppose that $\mathcal{C}$ is~a~root class of~groups (i.e.,~a~class of~groups that contains non-triv\-ial groups and~is~closed under taking subgroups and~unrestricted wreath products), $G$~is~the~free product of~residually $\mathcal{C}$\nobreakdash-groups~$A_{i}$ ($i \in \mathcal{I}$), and~$H$ is~a~subgroup of~$G$ satisfying a~non-triv\-ial identity. We~prove a~criterion for~the~$\mathcal{C}$\nobreakdash-sepa\-ra\-bil\-ity of~$H$ in~$G$. It~follows from~this criterion that, if~$\{\mathcal{V}_{j} \mid j \in \mathcal{J}\}$ is~a~family of~group varieties, each~$\mathcal{V}_{j}$ ($j \in \mathcal{J}$) is~distinct from~the~variety of~all groups, and~$\mathcal{V} = \bigcup_{j \in \mathcal{J}} \mathcal{V}_{j}$, then one can give a~description of~$\mathcal{C}$\nobreakdash-sep\-a\-ra\-ble $\mathcal{V}$\nobreakdash-sub\-groups of~$G$ provided such a~description is~known for~every group~$A_{i}$ ($i \in \mathcal{I}$).
\end{abstract}


\thanks{The~study was supported by~the~Russian Science Foundation grant No.~24-21-00307,\\ \url{http://rscf.ru/en/project/24-21-00307/}}

\maketitle

\section{Introduction. Statement of~results}\label{es01}

Let $\mathcal{C}$ be~a~class of~groups. Following~\cite{Malcev1958PIvPI}, we say that a~subgroup~$Y$ of~a~group~$X$ is~\emph{$\mathcal{C}$\nobreakdash-sep\-a\-ra\-ble} in~this group if, for~any element $x \in X \setminus Y$, there exists a~homomorphism~$\sigma$ of~$X$ onto~a~group from~$\mathcal{C}$ such that $x\sigma \notin Y\sigma$. Usually, the~$\mathcal{C}$\nobreakdash-sepa\-ra\-bil\-ity of~subgroups of~$X$ is~studied under the~assumption that $X$ is~\emph{residually a~$\mathcal{C}$\nobreakdash-group}. Recall that the~latter property is~equivalent to~the~$\mathcal{C}$\nobreakdash-sepa\-ra\-bil\-ity of~the~trivial subgroup.

It~is~well known that, if~the~class~$\mathcal{C}$ consists only of~finite groups and~the~group~$X$ is~finitely presented, then the~$\mathcal{C}$\nobreakdash-sepa\-ra\-bil\-ity of~the~subgroup~$Y$ implies the~existence of~an~algorithm that answers the~question of~whether a~given element of~$X$ belongs to~$Y$~\cite{Malcev1958PIvPI}. In~addition, the~$\mathcal{C}$\nobreakdash-sepa\-ra\-bil\-ity of~certain subgroups of~$X$ can sometimes serve as~one of~the~necessary and/or~sufficient conditions for~$X$ to~be~residually a~$\mathcal{C}$\nobreakdash-group. This relationship is~especially often found when $X$ is~a~group-theoretic construction and~$\mathcal{C}$ is~a~root class of~groups (see, for~example,~\cite{Loginova1999SMJ, Azarov2013MN, Tumanova2014MAIS, Tumanova2015IVM, Azarov2016SMJ, Sokolov2021SMJ2, Sokolov2022CA, SokolovTumanova2023SMJ}).

The~notion of~a~root class was~introduced in~\cite{Gruenberg1957PLMS}, and~the~equivalent definitions of~such a~class were~given in~\cite{Sokolov2015CA}. In~accordance with~one of~them, the~class~$\mathcal{C}$ is~called a~\emph{root class} if~it~contains non-triv\-ial groups and~is~closed under taking subgroups and~unrestricted wreath products. Examples of~root classes are~the~classes of~all finite groups, finite $p$\nobreakdash-groups (where $p$ is~a~prime number), periodic $\mathfrak{P}$\nobreakdash-groups of~finite exponent (where $\mathfrak{P}$ is~a~non-empty set of~primes), all solvable groups, and~all torsion-free groups. It~is~also easy to~show that, if~the~intersection of~a~family of~root classes contains a~non-triv\-ial group, then it~is~again a~root class.

Throughout the~paper, let $\mathfrak{P}(\mathcal{C})$ denote the~set of~prime numbers defined as~follows. If~the~class~$\mathcal{C}$ contains a~non-pe\-ri\-od\-ic group, then $\mathfrak{P}(\mathcal{C})$ is~the~set of~all primes. Otherwise, a~prime~$p$ belongs to~$\mathfrak{P}(\mathcal{C})$ if and~only if it~divides the~order of~an~element of~some $\mathcal{C}$\nobreakdash-group. Recall that the~subgroup~$Y$ of~the~group~$X$ is~said to~be~\emph{$\mathfrak{P}(\mathcal{C})^{\prime}$\nobreakdash-iso\-lat\-ed} in~$X$ if, for~each element $x \in X$ and~for each prime $q \notin \mathfrak{P}(\mathcal{C})$, the~inclusion $x^{q} \in Y$ means that $x \in Y$. It~is~clear that, if~$\mathfrak{P}(\mathcal{C})$ contains all prime numbers, then any subgroup is~$\mathfrak{P}(\mathcal{C})^{\prime}$\nobreakdash-iso\-lat\-ed.

Proposition~\ref{ep22} below asserts that every $\mathcal{C}$\nobreakdash-sep\-a\-ra\-ble subgroup of~$X$ is~$\mathfrak{P}(\mathcal{C})^{\prime}$\nobreakdash-iso\-lat\-ed. If~the~group~$X$ is~given by~generators and~defining relations, then the~latter property is~usually easier to~verify than the~$\mathcal{C}$\nobreakdash-sepa\-ra\-bil\-ity. Therefore, to~get a~criterion for~the~$\mathcal{C}$\nobreakdash-sepa\-ra\-bil\-ity of~subgroups of~$X$, it~suffices to~find a~description of~subgroups each of~which is~$\mathfrak{P}(\mathcal{C})^{\prime}$\nobreakdash-iso\-lat\-ed but~not~$\mathcal{C}$\nobreakdash-sep\-a\-ra\-ble in~$X$. Let~us call such subgroups \emph{$\mathcal{C}$\nobreakdash-de\-fec\-tive}.

In~this paper, for~a~given root class of~groups~$\mathcal{C}$ and~for the~(ordinary) free product~$G$ of~a~family of~groups, we study the~$\mathcal{C}$\nobreakdash-sepa\-ra\-bil\-ity of~subgroups of~$G$. A~free product almost always contains a~non-abelian free subgroup and~becomes a~free group when its factors are~all infinite cyclic. Therefore, any criterion for~the~$\mathcal{C}$\nobreakdash-sepa\-ra\-bil\-ity of~subgroups of~$G$ is~necessarily a~generalization of~the~corresponding result on~free groups.

It~is~known that, if~$\mathcal{C}$ is~a~root class of~groups, then any $\mathfrak{P}(\mathcal{C})^{\prime}$\nobreakdash-iso\-lat\-ed cyclic subgroup of~a~free group is~$\mathcal{C}$\nobreakdash-sep\-a\-ra\-ble in~this group~\cite[Proposition~5.3]{Sokolov2024SMJ}. If~we replace here the~word ``cyclic'' by~``finitely generated'', then the~resulting assertion does~not hold: the~corresponding counterexample is~given in~\cite{Bardakov2004SMJ}. Moreover, at~present, a~criterion for~the~$\mathcal{C}$\nobreakdash-sepa\-ra\-bil\-ity of~an~arbitrary finitely generated subgroup of~a~free group is~known only when $\mathcal{C}$ is~the~class of~all finite groups~\cite{Hall1949TAMS}. Therefore, when studying free products of~groups, it~makes sense to~impose additional conditions on~the~considered subgroups, which in~the~case of~a~free group are~equivalent to~the~property of~being cyclic. In~this article, such a~condition is~``to~satisfy a~non-triv\-ial identity''.

The~following theorem is~the~main result of~the~paper. For~a~given root class of~groups~$\mathcal{C}$ and~for~the~free product~$G$ of~a~family of~residually $\mathcal{C}$\nobreakdash-groups, it~completely solves the~question of~the~$\mathcal{C}$\nobreakdash-sepa\-ra\-bil\-ity of~a~subgroup of~$G$ that satisfies a~non-triv\-ial identity. It~makes sense to~note here that, by~Proposition~\ref{ep25} below, the~free product~$G$ is~residually a~$\mathcal{C}$\nobreakdash-group if and~only if all its factors have this property.

\begin{etheorem}
Suppose that $\mathcal{C}$ is~a~root class of~groups\textup{,} $G$~is~the~free product of~residually $\mathcal{C}$\nobreakdash-groups~$A_{i}$ \textup{(}$i \in \mathcal{I}$\textup{),} and~$H$ is~a~subgroup of~$G$ satisfying a~non-triv\-ial identity. Then $H$ is~$\mathcal{C}$\nobreakdash-de\-fec\-tive in~$G$ if and~only if it~is~conjugate to~a~$\mathcal{C}$\nobreakdash-de\-fec\-tive subgroup of~the~group~$A_{i}$ for~some $i \in \mathcal{I}$.
\end{etheorem}

Thus, to~get a~description of~$\mathcal{C}$\nobreakdash-sep\-a\-ra\-ble subgroups of~$G$ which satisfy a~non-triv\-ial identity, it~is~sufficient to~have such a~description for~every free factor~$A_{i}$ ($i \in \mathcal{I}$). The~next two corollaries follow directly from~the~above theorem.

\begin{ecorollary}\label{ec01}
Suppose that $\mathcal{C}$ is~a~root class of~groups and~$G$ is~the~free product of~residually $\mathcal{C}$\nobreakdash-groups~$A_{i}$ \textup{(}$i \in \mathcal{I}$\textup{)}. Suppose also that $\{\mathcal{V}_{j} \mid j \in \mathcal{J}\}$ is~a~family of~varieties of~groups\textup{,} each~$\mathcal{V}_{j}$ \textup{(}$j \in \mathcal{J}$\textup{)} is~distinct from~the~variety of~all groups\textup{,} and~$\mathcal{V} = \bigcup_{j \in \mathcal{J}} \mathcal{V}_{j}$. Then a~$\mathcal{V}$\nobreakdash-sub\-group~$H$ of~the~group $G$ is~$\mathcal{C}$\nobreakdash-de\-fec\-tive in~this group if and~only if it~is~conjugate to~a~$\mathcal{C}$\nobreakdash-de\-fec\-tive $\mathcal{V}$\nobreakdash-sub\-group of~the~group~$A_{i}$ for~some $i \in \mathcal{I}$. In~particular\textup{,} if~every group~$A_{i}$ \textup{(}$i \in \mathcal{I}$\textup{)} has the~property of~$\mathcal{C}$\nobreakdash-sepa\-ra\-bil\-ity of~all $\mathfrak{P}(\mathcal{C})^{\prime}$\nobreakdash-iso\-lat\-ed $\mathcal{V}$\nobreakdash-sub\-groups\textup{,} then the~free product~$G$ also has this property.
\end{ecorollary}

\begin{ecorollary}\label{ec02}
Suppose that $\mathcal{C}$ is~a~root class of~groups and~$G$ is~the~free product of~residually $\mathcal{C}$\nobreakdash-groups~$A_{i}$ \textup{(}$i \in \mathcal{I}$\textup{)}. Then a~cyclic subgroup~$H$ of~the~group~$G$ is~$\mathcal{C}$\nobreakdash-de\-fec\-tive in~this group if and~only if it~is~conjugate to~a~$\mathcal{C}$\nobreakdash-de\-fec\-tive cyclic subgroup of~the~group~$A_{i}$ for~some $i \in \mathcal{I}$. In~particular\textup{,} if~every group~$A_{i}$ \textup{(}$i \in \mathcal{I}$\textup{)} has the~property of~$\mathcal{C}$\nobreakdash-sepa\-ra\-bil\-ity of~all $\mathfrak{P}(\mathcal{C})^{\prime}$\nobreakdash-iso\-lat\-ed cyclic subgroups\textup{,} then the~free product~$G$ also has this property.
\end{ecorollary}

Let~us note that Corollary~\ref{ec02} generalizes Theorem~5 from~\cite{Stebe1968CPAM}, which asserts that the~free product of~$\pi_{c}$\nobreakdash-groups is~again a~$\pi_{c}$\nobreakdash-group (recall that a~\emph{$\pi_{c}$\nobreakdash-group} is~a~group whose cyclic subgroups are~all $\mathcal{F}$\nobreakdash-sep\-a\-ra\-ble, where $\mathcal{F}$ is~the~class of~all finite groups). The~analogue of~this result for~finitely generated subgroups was~get in~\cite{Romanovskii1969MUSSRI}, and~Corollary~\ref{ec01} complements~it in~some sense. The~proof of~the~formulated theorem is~given in~Section~\ref{es03}. Section~\ref{es02} contains a~number of~assertions necessary for~this proof.

\section{Some auxiliary concepts and~results}\label{es02}

\vspace*{2pt}

Given a~class of~groups~$\mathcal{C}$ and~a~group~$X$, let~us denote by~$\mathcal{C}^{*}(X)$ the~family of~normal subgroups of~$X$ defined as~follows: $N \in \mathcal{C}^{*}(X)$ if and~only if $X/N \in \mathcal{C}$. It~is~obvious that the~elements of~$\mathcal{C}^{*}(X)$ are~precisely the~kernels of~all possible homomorphisms of~$X$ onto~groups from~$\mathcal{C}$. Therefore, a~subgroup~$Y$ of~the~group~$X$ is~$\mathcal{C}$\nobreakdash-sep\-a\-ra\-ble in~this group if and~only if, for~each element $x \in X \setminus Y$, there exists a~subgroup $N \in \mathcal{C}^{*}(X)$ such that $x \notin YN$.

\begin{eproposition}\label{ep21}
\textup{\cite[Proposition~2]{SokolovTumanova2020IVM}}
If~$\mathcal{C}$ is~a~class of~groups closed under taking subgroups and~direct products of~a~finite number of~factors\textup{,} $X$~is~a~group\textup{,} and~$Y,Z \in \mathcal{C}^{*}(X)$\textup{,} then $Y \cap Z \in \mathcal{C}^{*}(X)$.
\end{eproposition}

\begin{eproposition}\label{ep22}
Suppose that $\mathcal{C}$ is~an~arbitrary class of~groups\textup{,} $X$~is~a~group\textup{,} and~$Y$ is~a~subgroup of~$X$. Then the~following statements hold.

\textup{1.\hspace{1ex}}If~the~subgroup~$Y$ is~$\mathcal{C}$\nobreakdash-sep\-a\-ra\-ble in~$X$\textup{,} then it~is~$\mathfrak{P}(\mathcal{C})^{\prime}$\nobreakdash-iso\-lat\-ed in~this group.

\textup{2.\hspace{1ex}}If~the~subgroup~$Y$ is~$\mathfrak{P}(\mathcal{C})^{\prime}$\nobreakdash-iso\-lat\-ed in~$X$ and~$Z$ is~a~$\mathfrak{P}(\mathcal{C})^{\prime}$\nobreakdash-iso\-lat\-ed subgroup of~$Y$\textup{,} then $Z$ is~$\mathfrak{P}(\mathcal{C})^{\prime}$\nobreakdash-iso\-lat\-ed in~$X$.
\end{eproposition}

\begin{proof}
The~first statement coincides with~Proposition~4.1 from~\cite{Sokolov2024SMJ}, while the~second one follows easily from~the~definition of~a~$\mathfrak{P}(\mathcal{C})^{\prime}$\nobreakdash-iso\-lat\-ed subgroup.
\end{proof}

Recall that a~subgroup~$Y$ of~a~group~$X$ is~said to~be~a~\emph{retract} of~this group if there exists a~homomorphism $\sigma\colon X \to Y$ acting identically on~$Y$. This homomorphism is~usually referred to~as~a~\emph{retracting} one.

\begin{eproposition}\label{ep23}
\textup{\cite[Proposition~3]{SokolovTumanova2016SMJ}}
Let $\mathcal{C}$ be~a~class of~groups closed under taking subgroups and~extensions. If~$X$ is~a~residually $\mathcal{C}$\nobreakdash-group\textup{,} then any retract of~$X$ is~a~$\mathcal{C}$\nobreakdash-sep\-a\-ra\-ble subgroup.
\end{eproposition}

In~what follows, let $\operatorname{sgp}\{S\}$ denote the~subgroup generated by~a~set~$S$.

\begin{eproposition}\label{ep24}
If~$G$ is~the~free product of~groups~$A_{i}$ \textup{(}$i \in \mathcal{I}$\textup{),} then\textup{,} for~each non-empty subset $\mathcal{J} \subseteq \mathcal{I}$\textup{,} the~subgroup $A_{\mathcal{J}} = \operatorname{sgp}\{A_{j} \mid j \in \mathcal{J}\}$ is~a~free factor of~$G$\textup{,} serves as~a~retract of~this group\textup{,} and~splits into~the~free product of~the~groups~$A_{j}$ \textup{(}$j \in \mathcal{J}$\textup{)}.
\end{eproposition}

\begin{proof}
Consider the~mapping of~the~generators of~$G$ to~$A_{\mathcal{J}}$ that acts identically on~the~generators of~the~groups~$A_{j}$ ($j \in \mathcal{J}$) and~takes the~generators of~the~groups~$A_{i}$ ($i \in \mathcal{I} \setminus \mathcal{J}$) to~the~trivial element. Obviously, this mapping can be~extended to~a~homomorphism, which is~identical on~the~subgroup~$A_{\mathcal{J}}$. Therefore, the~latter is~a~retract of~$G$. The~other two statements can be~easily proved by~grouping the~generators and~relations in~the~presentation of~the~group~$G$.
\end{proof}

Let $\mathcal{C}$ be~a~root class of~groups. It~is~easy to~see that any subgroup of~a~residually $\mathcal{C}$\nobreakdash-group is~also residually a~$\mathcal{C}$\nobreakdash-group. Therefore, the~next proposition completely solves the~question of~whether the~free product of~a~family of~groups is~residually a~$\mathcal{C}$\nobreakdash-group.

\begin{eproposition}\label{ep25}
\textup{\cite[Theorem~2]{AzarovTieudjo2002PIvSU}}
If~$\mathcal{C}$ is~a~root class of~groups\textup{,} then the~free product of~any family of~residually $\mathcal{C}$\nobreakdash-groups is~itself residually a~$\mathcal{C}$\nobreakdash-group.
\end{eproposition}

Let~$G$ be~the~free product of~groups~$A$ and~$B$. Recall (see, for~example,~\cite[Theorem~4.1]{MagnusKarrassSolitar1974}) that each element $g \in G$ can be uniquely represented as~a~product $1 \,{\cdot}\, g_{1}g_{2} \ldots g_{n}$, where $n \geqslant 0$, $g_{1}, g_{2}, \ldots, g_{n} \in (A \cup B) \setminus \{1\}$, and,~for~$n > 1$, no~two adjacent factors~$g_{i}$ and~$g_{i+1}$ lie simultaneously in~$A$ or~in~$B$. This product is~referred to~as~the~\emph{reduced form} of~the~element~$g$, and~the~number~$n$ is~said to~be~the~\emph{length} of~this element. The~latter is~denoted below by~$\ell(g)$.

\pagebreak

Obviously, every free product of~two groups can be~considered as~the~generalized free product of~these groups with~the~trivial amalgamated subgroup (to~avoid ambiguity, we always refer to~a~free product with~an~amalgamated subgroup as~the~\emph{generalized} one). Therefore, the~next proposition is~a~special case of~Theorem~2.1 from~\cite{Sokolov2024SMJ}.

\begin{eproposition}\label{ep26}
Suppose that $\mathcal{C}$ is~a~root class of~groups and~$G$ is~the~free product of~groups~$A$ and~$B$\kern-.5pt{}. Suppose also that $G$\kern-.5pt{} has a~homomorphism which maps~it onto~a~$\mathcal{C}$\nobreakdash-group and~acts injectively on~$A$ and~$B$. If~a~$\mathfrak{P}(\mathcal{C})^{\prime}$\nobreakdash-iso\-lat\-ed subgroup of~$G$ satisfies a~non-triv\-ial identity\textup{,} then it~is~$\mathcal{C}$\nobreakdash-sep\-a\-ra\-ble in~$G$. In~particular\textup{,} $G$~is~residually a~$\mathcal{C}$\nobreakdash-group.
\end{eproposition}

\begin{eproposition}\label{ep27}
\textup{\cite[Theorem~7]{KarrassSolitar1970TAMS}}
Let~$G$ be~the~generalized free product of~groups~$A$ and~$B$ with~an~amalgamated subgroup~$U$. If~a~subgroup~$H$ of~$G$ satisfies a~non-triv\-ial identity\textup{,} then it~is~one of~the~following\textup{:}

\textup{a\kern.5pt)\hspace{1ex}}a~subgroup of~a~conjugate of~$A$ or~$B$\textup{;}

\textup{b)\hspace{1ex}}the~union of~a~countable non-de\-scend\-ing sequence of~subgroups\textup{,} each of~which has the~form $gUg^{-1} \cap H$ for~some $g \in G$\textup{;}

\textup{c\kern1.2pt)\hspace{1ex}}an~extension of~the~above union by~an~infinite cyclic group\textup{;}

\textup{d)\hspace{1ex}}the~generalized free product of~groups~$\tilde A$ and~$\tilde B$ with~an~amalgamated subgroup~$\tilde U$\textup{,} where $\tilde A, \tilde B \in \{g_{a}^{\vphantom{1}}Ag_{a}^{-1} \cap H \mid g_{a}^{\vphantom{1}} \in G\} \cup \{g_{b}^{\vphantom{1}}Bg_{b}^{-1} \cap H \mid g_{b}^{\vphantom{1}} \in G\}$\textup{,} $\tilde U = gUg^{-1} \cap H$ for~some $g \in G$\textup{,} and~$[\tilde A : \tilde U] = 2 = [\tilde B : \tilde U]$.
\end{eproposition}

\begin{eproposition}\label{ep28}
Let~$G$ be~the~free product of~groups~$A$ and~$B$. If~a~subgroup~$H$ of~$G$ satisfies a~non-triv\-ial identity\textup{,} then one and~only one of~the~following statements holds.{\parfillskip=0pt\par}

\textup{1.\hspace{1ex}}The~subgroup~$H$ is~conjugate to~a~subgroup of~$A$ or~of~$B$.

\textup{2.\hspace{1ex}}The~subgroup~$H$ is~infinite cyclic and~is~conjugate to~no~subgroup lying in~$A \cup B$.

\textup{3.\hspace{1ex}}The~subgroup~$H$\kern-1.5pt{} is~the~non-abelian split extension of~an~infinite cyclic group by~a~group of~order~$2$ and~is~conjugate to~no~subgroup lying in~$A \cup B$.
\end{eproposition}

\begin{proof}
Indeed, let~us consider the~group~$G$ as~the~generalized free product of~the~groups~$A$ and~$B$ with~the~trivial amalgamated subgroup~$U$. Then, by~Proposition~\ref{ep27}, the~subgroup~$H$ is~one of~the~following:

a\kern.5pt)\hspace{1ex}a~subgroup of~a~conjugate of~$A$ or~$B$;

b)\hspace{1ex}a~trivial group;

c\kern1.2pt)\hspace{1ex}an~infinite cyclic group;

d)\hspace{1ex}the~free product of~groups~$\tilde A$ and~$\tilde B$ of~order~$2$.

Suppose that the~last possibility is~realized and~the~symbols~$a$ and~$b$ denote the~generators of~the~groups~$\tilde A$ and~$\tilde B$, respectively. Since, for~each $k \geqslant 1$, the~element~$(ab)^{k}$ has a~reduced form of~length~$2k$ in~the~free product~$H$, the~cyclic subgroup~$Z$ generated by~$ab$ is~infinite. It~follows from~the~equalities $a^{-1} = a$, $b^{-1} = b$, $(ab)^{-1} = ba$, and~$a^{-1}(ab)a = ba = b^{-1}(ab)b$ that the~subgroup~$Z$ is~normal in~$H$, but~does~not lie in~the~center of~this group. It~is~also obvious that $H = \tilde AZ$ and~$\tilde A \cap Z = 1$. Hence, $H$~is~the~non-abelian split extension of~the~infinite cyclic group~$Z$ by~the~group~$\tilde A$ of~\mbox{order}~$2$.
\end{proof}

Given a~class of~groups~$\mathcal{C}$, let~us say that a~group~$X$ is~\emph{$\mathcal{C}$\nobreakdash-quasi-regu\-lar} with~respect to~its subgroup~$Y$ if, for~each subgroup $M \in \mathcal{C}^{*}(Y)$, there exists a~subgroup $N \in \mathcal{C}^{*}(X)$ such that $N \cap Y \leqslant M$. The~next proposition is~a~special case of~Corollary~2.4 from~\cite{Sokolov2024SMJ}.{\parfillskip=0pt\par}

\begin{eproposition}\label{ep29}
Suppose that $\mathcal{C}$ is~a~root class of~groups\textup{,} $G$~is~the~generalized free product of~residually $\mathcal{C}$\nobreakdash-groups~$A$ and~$B$ with~an~amalgamated subgroup~$U$\textup{,} and~$H$ is~a~cyclic subgroup of~$G$ which is~conjugate to~no~subgroup lying in~$A \cup B$. Suppose also that $U$ is~$\mathcal{C}$\nobreakdash-sepa\-ra\-ble in~the~groups~$A$ and~$B$\textup{,} while $G$ is~$\mathcal{C}$\nobreakdash-quasi-regu\-lar with~respect to~the~subgroups~$A$ and~$B$. If~the~subgroup~$H$ is~$\mathfrak{P}(\mathcal{C})^{\prime}$\nobreakdash-iso\-lat\-ed in~$G$\textup{,} then it~is~$\mathcal{C}$\nobreakdash-sep\-a\-ra\-ble in~this~group.{\parfillskip=0pt\par}
\end{eproposition}

\begin{eproposition}\label{ep20}
Suppose that $\mathcal{C}$ is~a~root class of~groups and~$G$ is~the~free product of~residually $\mathcal{C}$\nobreakdash-groups~$A$ and~$B$. Suppose also that $H$ is~a~cyclic subgroup of~$G$ which is~conjugate to~no~subgroup lying in~$A \cup B$. If~$H$ is~$\mathfrak{P}(\mathcal{C})^{\prime}$\nobreakdash-iso\-lat\-ed in~$G$\textup{,} then it~is~$\mathcal{C}$\nobreakdash-sep\-a\-ra\-ble in~this group.
\end{eproposition}

\begin{proof}
Suppose that $M \in \mathcal{C}^{*}(A)$ and~$\varepsilon\colon A \to A/M$ is~the~natural homomorphism. Since $A$ is~a~retract of~$G$ by~Proposition~\ref{ep24}, there exists a~homomorphism $\sigma\colon G \to A$ which acts identically on~$A$. If~$N = \ker\sigma\varepsilon$, then $N \in \mathcal{C}^{*}(G)$ and~$N \cap A = M$, as~it~is~easy to~see. Therefore, the~group~$G$ is~$\mathcal{C}$\nobreakdash-quasi-regu\-lar with~respect to~$A$. Its~$\mathcal{C}$\nobreakdash-quasi-regu\-lar\-i\-ty with~respect to~$B$ can be~proved similarly. Since $A$ and~$B$ are~residually $\mathcal{C}$\nobreakdash-groups, their trivial subgroups are~$\mathcal{C}$\nobreakdash-sep\-a\-ra\-ble. Thus, the~desired assertion follows from~Proposition~\ref{ep29} if $G$ is~considered as~the~generalized free product of~the~groups~$A$ and~$B$ with~the~trivial amalgamated subgroup.
\end{proof}

\section{Proof of~Theorem}\label{es03}

\begin{eproposition}\label{ep31}
Suppose that $\mathcal{C}$ is~a~root class of~groups and~$G$ is~the~free product of~residually $\mathcal{C}$\nobreakdash-groups~$A$ and~$B$. Suppose also that $H$ is~a~subgroup of~$G$\textup{,} this subgroup is~the~non-abelian split extension of~an~infinite cyclic group by~a~group of~order~$2$ and~is~conjugate to~no~subgroup lying in~$A \cup B$. If~$H$ is~$\mathfrak{P}(\mathcal{C})^{\prime}$\nobreakdash-iso\-lat\-ed in~$G$\textup{,} then it~is~$\mathcal{C}$\nobreakdash-sep\-a\-ra\-ble in~this group.
\end{eproposition}

\begin{proof}
As~it~is~known (see, for~example,~\cite[Corollary~4.1.4]{MagnusKarrassSolitar1974}), every element of~finite order of~the~group~$G$ is~conjugate to~an~element of~$A$ or~of~$B$. Therefore, we may replace the~subgroup~$H$ by~its conjugate, rename the~free factors~$A$ and~$B$, if~necessary, and~assume further that $H$ is~the~non-abelian split extension of~an~infinite cyclic subgroup $Z \leqslant G$ by~a~subgroup $Y \leqslant A$ of~order~$2$. Let $Y = \{1,y\}$, and~let $z$ be~a~generator of~$Z$. Then $y^{-1}zy = z^{-1}$. We~claim that the~subgroup~$Z$ is~conjugate to~no~subgroup contained in~$A \cup B$.

Let, on~the~contrary, $x^{-1}Zx \subseteq A \cup B$ for~some $x \in G$, and~let $1 \,{\cdot}\, x_{1}x_{2} \ldots x_{n}$ be~the~reduced form of~$x$. We~will argue by~induction on~$n$. If~$n = 0$ and~$Z \leqslant A$, then $H = YZ \leqslant A$ despite the~condition of~the~proposition. If~$n = 0$ and~$Z \leqslant B$, then the~product $1 \,{\cdot}\, y^{-1}zyz$ is~the~reduced form of~the~element~$y^{-1}zyz$, and~we get a~contradiction with~the~equality $y^{-1}zy = z^{-1}$.

Suppose that $n \geqslant 1$, $x^{-1}Zx \leqslant A$, and~hence $z = x_{1}^{\vphantom{1}}x_{2}^{\vphantom{1}} \ldots x_{n}^{\vphantom{1}}ax_{n}^{-1} \ldots x_{2}^{-1}x_{1}^{-1}$ for~some $a \in\nolinebreak A \setminus \{1\}$. If~$x_{n} \in A$ and~$x^{\prime}= x_{1}x_{2} \ldots x_{n-1}$, then $(x^{\prime})^{-1}Zx^{\prime} \leqslant A$ and~the~inductive hypothesis can be~used. Let $x_{n} \in B$. In~this case, the~product $1 \,{\cdot}\, x_{1}^{\vphantom{1}}x_{2}^{\vphantom{1}} \ldots x_{n}^{\vphantom{1}}a^{\pm 1}_{\vphantom{1}}x_{n}^{-1} \ldots x_{2}^{-1}x_{1}^{-1}$ is~the~reduced form of~the~element~$z^{\pm 1}$, and~this element is~of~length~$2n+1$. If~$x_{1} \in B$, we have $\ell(y^{-1}zy) = 2n+3 > 2n+1 = \ell(z^{-1})$, which is~impossible. Since the~reduced form of~the~element $y^{-1}zy = z^{-1}$ is~unique, it~follows from~the~inclusion $x_{1} \in A$ that $y^{-1}x_{1} = x_{1}$ and~$y = 1$, which is~also impossible. The~same arguments can be~used in~the~case where $n \geqslant 1$ and~$x^{-1}Zx \leqslant B$.

Thus, $Z$ is~conjugate to~no~subgroup contained in~$A \cup B$. Since $A$ is~residually a~$\mathcal{C}$\nobreakdash-group, the~inclusion $Y \leqslant A$ implies that $2 \in \mathfrak{P}(\mathcal{C})$. Therefore, $Z$ is~$\mathfrak{P}(\mathcal{C})^{\prime}$\nobreakdash-iso\-lat\-ed in~the~subgroup~$H$. The~latter is~$\mathfrak{P}(\mathcal{C})^{\prime}$\nobreakdash-iso\-lat\-ed in~$G$. Hence, $Z$~is~also $\mathfrak{P}(\mathcal{C})^{\prime}$\nobreakdash-iso\-lat\-ed in~$G$ by~Proposition~\ref{ep22} and~is~$\mathcal{C}$\nobreakdash-sep\-a\-ra\-ble in~this group by~Proposition~\ref{ep20}.

To~prove the~$\mathcal{C}$\nobreakdash-sepa\-ra\-bil\-ity of~$H$, let~us fix an~element $g \in G \setminus H$ and~find a~subgroup $N \in\nolinebreak \mathcal{C}^{*}(G)$ satisfying the~relation $g \notin HN$. Since $g \notin H = Z \cup yZ$, we have $g \notin Z$ and~$y^{-1}g \notin Z$. It~follows from~the~$\mathcal{C}$\nobreakdash-sepa\-ra\-bil\-ity of~$Z$ and~Proposition~\ref{ep21} that there exists a~subgroup $N \in \mathcal{C}^{*}(G)$ such that $g \notin ZN$ and~$y^{-1}g \notin ZN$. If~$g \in HN = ZN \cup\nolinebreak yZN$, then $g \in ZN$ or~$y^{-1}g \in ZN$ despite the~choice of~$N$. Therefore, the~subgroup~$N$ is~the~desired~one.
\end{proof}

\begin{eproposition}\label{ep32}
Suppose that $\mathcal{C}$ is~a~root class of~groups and~$G$ is~the~free product of~residually $\mathcal{C}$\nobreakdash-groups~$A$ and~$B$. Suppose also that $H$ is~a~subgroup of~$G$ which is~contained in~$A \cup B$ and~satisfies a~non-triv\-ial identity. If~an~element $g \in G$ and~the~subgroup~$H$ do~not lie in~the~same free factor\textup{,} then there exists a~homomorphism~$\sigma$ of~$G$ onto~a~group from~$\mathcal{C}$ such that $g\sigma \notin H\sigma$.
\end{eproposition}

\begin{proof}
Without lost of~generality, we may assume that $H \leqslant A$. Let $1 \,{\cdot}\, g_{1}g_{2} \ldots g_{n}$ be~the~reduced form of~$g$. Then $n \geqslant 1$ and~$g_{1} \in B$ if~$n = 1$. For~each $i \in \{1, 2, \ldots, n\}$, let~us define subgroups $M_{i}$ and~$N_{i}$ as~follows. If~$g_{i} \in A$, we put $N_{i} = B$ and~find a~subgroup $M_{i} \in \mathcal{C}^{*}(A)$ satisfying the~relation $g_{i} \notin M_{i}$ (such a~subgroup certainly exists because $A$ is~residually a~$\mathcal{C}$\nobreakdash-group). Similarly, if~$g_{i} \in B$, we take~$A$ as~$M_{i}$ and~choose a~subgroup $N_{i} \in \mathcal{C}^{*}(B)$ which does~not contain~$g_{i}$.

Let $M = \bigcap_{i=1}^{n} M_{i}$ and~$N = \bigcap_{i=1}^{n} N_{i}$. Then $A/M,\,B/N \in \mathcal{C}$ by~Proposition~\ref{ep21} and,~for~any $i \in \{1, 2, \ldots, n\}$, $g_{i}M \ne 1$ if~$g_{i} \in A$, and~$g_{i}N \ne 1$ if~$g_{i} \in B$. It~is~easy to~see that the~natural homomorphisms $A \to A/M$ and~$B \to B/N$ can be~extended to~a~homomorphism~$\rho$ of~the~group~$G$ onto~the~free product~$G_{M,N}$ of~the~groups~$A/M$ and~$B/N$. It~is~also obvious that the~product $1 \,{\cdot}\, (g_{1}\rho)(g_{2}\rho) \ldots (g_{n}\rho)$ is~the~reduced form of~the~element~$g\rho$ and~$g_{1}\rho \in B/N$ if~$n = 1$. Therefore, $g\rho \notin H\rho$.

Since $A\rho\kern-.5pt{} =\kern-1.5pt{} A/M\kern-2pt{} \in\kern-.5pt{} \mathcal{C}$ and~$B\kern-.5pt{}\rho\kern-.5pt{} =\kern-1.5pt{} B/N\kern-2pt{} \in\kern-.5pt{} \mathcal{C}$, the~subgroup\kern-.5pt{}~$H\kern-1.5pt{}\rho$ is~$\mathcal{C}$\nobreakdash-sep\-a\-ra\-ble in~the~group\kern-.5pt{}~$A\rho$ and,~by~Proposition~\ref{ep25}, the~free product~$G_{M,N}$ is~residually a~$\mathcal{C}$\nobreakdash-group. The~subgroup~$A\rho$ is~a~retract of~$G_{M,N}$ by~Proposition~\ref{ep24} and~is~$\mathcal{C}$\nobreakdash-sep\-a\-ra\-ble in~this group by~Proposition~\ref{ep23}. Hence, the~subgroup~$H\rho$ is~$\mathfrak{P}(\mathcal{C})^{\prime}$\nobreakdash-iso\-lat\-ed in~the~group~$A\rho$, the~subgroup~$A\rho$ is~$\mathfrak{P}(\mathcal{C})^{\prime}$\nobreakdash-iso\-lat\-ed in~the~group~$G_{M,N}$, and~the~subgroup~$H\rho$ is~$\mathfrak{P}(\mathcal{C})^{\prime}$\nobreakdash-iso\-lat\-ed in~the~group~$G_{M,N}$ by~Proposition~\ref{ep22}. It~is~also clear that the~identity mappings of~the~groups~$A\rho$ and~$B\rho$ induce a~homomorphism of~$G_{M,N}$ onto~the~direct product~$A\rho \times B\rho$, which belongs to~$\mathcal{C}$. Therefore, the~subgroup~$H\rho$ is~$\mathcal{C}$\nobreakdash-sep\-a\-ra\-ble in~the~group~$G_{M,N}$ by~Proposition~\ref{ep26} and the~homomorphism~$\rho$ can be~extended to~the~desired one.
\end{proof}

\begin{eproposition}\label{ep33}
Suppose that $\mathcal{C}$ is~a~root class of~groups\textup{,} $G$~is~the~free product of~residually $\mathcal{C}$\nobreakdash-groups~$A_{1}$\textup{,} $A_{2}$\textup{,}~\ldots\textup{,}~$A_{n}$ \textup{(}$n \geqslant 1$\textup{),} and~$H$ is~a~$\mathfrak{P}(\mathcal{C})^{\prime}$\nobreakdash-iso\-lat\-ed subgroup of~$G$ satisfying a~non-triv\-ial identity. If\textup{,}~for~any $i \in \{1, 2, \ldots, n\}$\textup{,} the~subgroup~$H$ is~conjugate to~no~$\mathcal{C}$\nobreakdash-de\-fec\-tive subgroup of~$A_{i}$\textup{,} then it~is~$\mathcal{C}$\nobreakdash-sep\-a\-ra\-ble in~$G$ and\textup{,} therefore\textup{,} is~not $\mathcal{C}$\nobreakdash-de\-fec\-tive in~this~group.
\end{eproposition}

\begin{proof}
We will argue by~induction on~$n$. The~proposition is~trivial if~$n = 1$, and~we assume further that $n > 1$. By~Proposition~\ref{ep24}, the~group~$G$ is~the~free product of~the~groups~$A_{n}$ and~$B = \operatorname{sgp}\{A_{1}, A_{2}, \ldots, A_{n-1}\}$, while the~latter is~the~free product of~the~groups~$A_{1}$, $A_{2}$,~\ldots,~$A_{n-1}$. Since~$A_{1}$, $A_{2}$,~\ldots,~$A_{n}$ are~residually $\mathcal{C}$\nobreakdash-groups, it~follows from~Proposition~\ref{ep25} that $B$ and~$G$ are~also residually $\mathcal{C}$\nobreakdash-groups. This fact allows one to~apply Propositions~\ref{ep28},~\ref{ep20}, and~\ref{ep31} to~the~free product of~the~groups~$A_{n}$ and~$B$. If~Statement~2 or~Statement~3 of~Proposition~\ref{ep28} holds, then $H$ is~$\mathcal{C}$\nobreakdash-sep\-a\-ra\-ble in~$G$. Therefore, we may replace this subgroup by~its conjugate, if~necessary, and~assume further that $H \leqslant A_{n}$ or~$H \leqslant B$. To~prove the~$\mathcal{C}$\nobreakdash-sepa\-ra\-bil\-ity of~$H$, we fix an~element $g \in G \setminus H$ and~find a~homomorphism~$\sigma$ of~the~group~$G$ onto~a~group from~$\mathcal{C}$ such that $g\sigma \notin H\sigma$. Consider three~cases.

\textit{Case\hspace{.8ex}1.}\hspace{1ex}$H \leqslant A_{n}$ and~$g \in A_{n}$.

By~Proposition~\ref{ep24}, the~free factor~$A_{n}$ is~a~retract of~$G$. Obviously, the~subgroup~$H$ is~$\mathfrak{P}(\mathcal{C})^{\prime}$\nobreakdash-iso\-lat\-ed in~the~group~$A_{n}$ and~is~conjugate in~$A_{n}$ to~no~$\mathcal{C}$\nobreakdash-de\-fec\-tive subgroup. Hence, it~is~not $\mathcal{C}$\nobreakdash-de\-fec\-tive in~$A_{n}$ and,~therefore, is~$\mathcal{C}$\nobreakdash-sep\-a\-ra\-ble in~this group. It~follows that the~retracting homomorphism $G \to A_{n}$, which acts identically on~the~subgroup~$A_{n}$, can be~extended to~the~desired mapping.

\pagebreak

\textit{Case\hspace{.8ex}2.}\hspace{1ex}$H \leqslant B$ and~$g \in B$.

Similarly, the~subgroup~$H$ is~$\mathfrak{P}(\mathcal{C})^{\prime}$\nobreakdash-iso\-lat\-ed in~$B$ and,~for~any $i \in \{1, 2, \ldots, n-1\}$, is~conjugate to~no~$\mathcal{C}$\nobreakdash-de\-fec\-tive subgroup of~the~group~$A_{i}$. Hence, it~is~$\mathcal{C}$\nobreakdash-sep\-a\-ra\-ble in~$B$ by~the~inductive hypothesis, and~the~retracting homomorphism $G \to B$ can be~extended to~the~desired one.

\textit{Case\hspace{.8ex}3.}\hspace{1ex}Either $H \leqslant A_{n}$ and~$g \notin A_{n}$, or~$H \leqslant B$ and~$g \notin B$.

In~this case, the~existence of~the~required homomorphism is~guaranteed by~Proposition~\ref{ep32}.
\end{proof}

\begin{proof}[\textup{\textbf{Proof of~Theorem.}}]
\textit{Necessity.} Let~us argue by~contradiction. If~the~subgroup~$H$ is~not $\mathfrak{P}(\mathcal{C})^{\prime}$\nobreakdash-iso\-lat\-ed in~$G$, then it~is~not $\mathcal{C}$\nobreakdash-de\-fec\-tive in~this group by~the~definition of~the~latter property. Therefore, we may assume that $H$ is~$\mathfrak{P}(\mathcal{C})^{\prime}$\nobreakdash-iso\-lat\-ed in~$G$ and,~for~each $i \in \mathcal{I}$, is~conjugate to~no~$\mathcal{C}$\nobreakdash-de\-fec\-tive subgroup of~$A_{i}$. Let~us show that $H$ is~then $\mathcal{C}$\nobreakdash-sep\-a\-ra\-ble in~$G$ and,~hence, is~not $\mathcal{C}$\nobreakdash-de\-fec\-tive in~this group. If~the~set~$\mathcal{I}$ is~finite, the~$\mathcal{C}$\nobreakdash-sepa\-ra\-bil\-ity of~$H$ follows from~Proposition~\ref{ep33}. Therefore, $\mathcal{I}$~can be~assumed to~be~infinite. As~above, to~prove the~$\mathcal{C}$\nobreakdash-sepa\-ra\-bil\-ity of~$H$, we fix an~element $g \in G \setminus H$ and~find a~homomorphism~$\sigma$ of~$G$ onto~a~group from~$\mathcal{C}$ which satisfies the~condition $g\sigma \notin H\sigma$.

Let $\mathcal{J}$ be~a~finite subset of~$\mathcal{I}$ such that $g \in \operatorname{sgp}\{A_{j} \mid j \in \mathcal{J}\}$. Let also the~symbols~$A$ and~$B$ denote the~subgroups $\operatorname{sgp}\{A_{j} \mid j \in \mathcal{J}\}$ and~$\operatorname{sgp}\{A_{i} \mid i \in \mathcal{I} \setminus \mathcal{J}\}$, respectively. By~Proposition~\ref{ep24}, the~group~$G$ splits into~the~free product of~the~groups~$A$ and~$B$, which in~turn are~the~free products of~the~subgroups generating them. It~follows that~$A$, $B$, and~$G$ are~residually $\mathcal{C}$\nobreakdash-groups by~Proposition~\ref{ep25}. Let~us apply Propositions~\ref{ep28},~\ref{ep20}, and~\ref{ep31} to~the~group~$G$ considered as~the~free product of~the~groups~$A$ and~$B$. As~in~the~proof of~Proposition~\ref{ep33}, if~Statement~2 or~Statement~3 of~Proposition~\ref{ep28} holds, then $H$ is~$\mathcal{C}$\nobreakdash-sep\-a\-ra\-ble in~$G$. If~$H$ is~conjugate to~a~subgroup of~$B$ and~$\rho\colon G \to A$ is~the~retracting homomorphism, then $g\rho = g \ne 1 = H\rho$ and,~since $A$ is~residually a~$\mathcal{C}$\nobreakdash-group, $\rho$~can be~extended to~the~desired mapping. Therefore, we may assume further that $x^{-1}Hx \leqslant A$ for~some $x \in G$ and~$1 \,{\cdot}\, x_{1}x_{2} \ldots x_{n}$ is~the~reduced form of~the~element~$x$ in~the~free product of~$A$ and~$B$.

Let~us use induction on~$n$. If~$n = 0$, then $H \leqslant A$ and~the~retracting homomorphism $G \to\nolinebreak A$ can be~extended to~the~desired one due~to~Proposition~\ref{ep33}. If~$n \geqslant 1$, $x_{n} \in\nolinebreak A$, and~$x^{\prime}= x_{1}x_{2} \ldots x_{n-1}$, then $(x^{\prime})^{-1}Hx^{\prime} \leqslant A$ and~the~required homomorphism exists by~the~inductive hypothesis. Suppose that $n \geqslant 1$, $x_{n} \in B$, $\tilde H = x^{-1}Hx$, and~$\tilde g = x^{-1}gx$. If~$x_{1} \in B$, then $\ell(\tilde g) = 2n+1 \geqslant 3$. If~$x_{1} \in A$, it~follows from~the~relations $x_{n} \in B$ and~$g \notin H$ that $n \geqslant 2$, $x_{1}^{-1}gx_{1}^{\vphantom{1}} \ne 1$, and~hence $\ell(\tilde g) = 2n-1 \geqslant 3$. Thus, in~both cases, $\tilde H \leqslant A$, $\tilde g \notin A$, and~the~existence of~the~required homomorphism is~ensured by~Proposition~\ref{ep32}.

\medskip

\textit{Sufficiency.} Suppose that $x^{-1}Hx \leqslant A_{i}$ for~some $x \in G$ and~$i \in \mathcal{I}$, and~the~subgroup $\tilde H = x^{-1}Hx$ is~$\mathcal{C}$\nobreakdash-de\-fec\-tive in~the~group~$A_{i}$. Then $\tilde H$ is~$\mathfrak{P}(\mathcal{C})^{\prime}$\nobreakdash-iso\-lat\-ed in~$A_{i}$ and~there exists an~element $a \in A_{i} \setminus \tilde H$ such that $a\sigma_{i} \in \tilde H\sigma_{i}$ for~every homomorphism~$\sigma_{i}$ of~the~group~$A_{i}$ onto~a~group from~$\mathcal{C}$. Let $\sigma$ be~an~arbitrary homomorphism of~$G$ onto~a~$\mathcal{C}$\nobreakdash-group. Since the~class~$\mathcal{C}$ is~closed under taking subgroups, the~restriction of~$\sigma$ to~the~subgroup~$A_{i}$ maps the~latter onto~a~group from~$\mathcal{C}$. Therefore, $a\sigma \in \tilde H\sigma$, $xax^{-1}\sigma \in H\sigma$, and,~since $xax^{-1} \notin H$, the~subgroup~$H$ is~not $\mathcal{C}$\nobreakdash-sep\-a\-ra\-ble in~$G$.

As~noted above, $G$~is~residually a~$\mathcal{C}$\nobreakdash-group. Hence, the~subgroup~$A_{i}$ is~$\mathcal{C}$\nobreakdash-sep\-a\-ra\-ble in~$G$ due~to~Propositions~\ref{ep24} and~\ref{ep23}. This fact and~Proposition~\ref{ep22} imply that the~subgroups~$A_{i}$, $\tilde H$, and~$H$ are~$\mathfrak{P}(\mathcal{C})^{\prime}$\nobreakdash-iso\-lat\-ed in~the~group~$G$. Thus, the~subgroup~$H$ is~$\mathcal{C}$\nobreakdash-de\-fec\-tive in~$G$, as~required.
\end{proof}


\newpage



\begin{thebibliography}{99}





\bibitem{Malcev1958PIvPI}
\textit{Mal'cev\hspace{.7ex}A.\hspace{.5ex}I.}\hspace{1ex}On~homomorphisms onto~finite groups, Ivanov.\ Gos.\ Ped.\ Inst.\ Ucen.\ Zap.\ 18 (1958) 49\nobreakdash--60 (in~Russian). See also: \textit{Mal'cev\hspace{.7ex}A.\hspace{.5ex}I.}\hspace{1ex}On~homomorphisms onto~finite groups, Transl.\ Am.\ Math.\ Soc.\ 2 (119) (1983) 67\nobreakdash--79.

\bibitem{Loginova1999SMJ}
\textit{Loginova\hspace{.7ex}E.\hspace{.5ex}D.}\hspace{1ex}Residual finiteness of~the~free product of~two groups with~commuting subgroups, Sib.\ Math.~J.\ 40 (2) (1999) 341\nobreakdash--350. DOI:~\href{https://doi.org/10.1007/s11202-999-0013-8}{10.1007/s11202-999-0013-8}.

\bibitem{Azarov2013MN}
\textit{Azarov\hspace{.7ex}D.\hspace{.5ex}N.}\hspace{1ex}On~the~residual finiteness of~free products of~solvable minimax groups with~cyclic amalgamated subgroups, Math.\ Notes 93 (4) (2013) 503\nobreakdash--509. DOI:~\href{https://doi.org/10.1134/S0001434613030188}{10.1134/S0001434613030188}.{\parfillskip=0pt\par}

\bibitem{Tumanova2014MAIS}
\textit{Tumanova\hspace{.7ex}E.\hspace{.5ex}A.}\hspace{1ex}On~the~root-class residuality of~HNN-ex\-ten\-sions of~groups, Model.\ Anal.\ Inform.\ Syst.\ 21 (4) (2014) 148\nobreakdash--180 (in~Russian). DOI:~\href{http://doi.org/10.18255/1818-1015-2014-4-148-180}{10.18255/1818-1015-2014-4-148-180}.

\bibitem{Tumanova2015IVM}
\textit{Tumanova\hspace{.7ex}E.\hspace{.5ex}A.}\hspace{1ex}On~the~root-class residuality of~generalized free products with~a~normal amalgamation, Russ. Math.\ 59 (10) (2015) 23\nobreakdash--37. DOI:~\href{https://doi.org/10.3103/S1066369X15100035}{10.3103/S1066369X15100035}.

\bibitem{Azarov2016SMJ}
\textit{Azarov\hspace{.7ex}D.\hspace{.5ex}N.}\hspace{1ex}A~criterion for~the~$\mathcal{F}_{\pi}$\nobreakdash-re\-sid\-u\-ality of~free products with~amalgamated cyclic subgroup of~nilpotent groups of~finite ranks, Sib.\ Math.~J.\ 57 (3) (2016) 377\nobreakdash--384.\newline DOI:~\href{https://doi.org/10.1134/S0037446616030010}{10.1134/S0037446616030010}.

\bibitem{Sokolov2021SMJ2}
\textit{Sokolov\hspace{.7ex}E.\hspace{.5ex}V.}\hspace{1ex}The~root-class residuality of~the~fundamental groups of~certain graph of~groups with central edge subgroups, Sib.\ Math.~J.\ 62 (6) (2021) 1119\nobreakdash--1132. DOI:~\href{http://dx.doi.org/10.1134/S0037446621060136}{10.1134/S0037446621060136}.{\parfillskip=0pt{}\par}

\bibitem{Sokolov2022CA}
\textit{Sokolov\hspace{.7ex}E.\hspace{.5ex}V.}\hspace{1ex}Certain residual properties of~HNN-ex\-ten\-sions with~central associated subgroups, Comm.\ Algebra 50 (3) (2022) 962\nobreakdash--987. DOI:~\href{http://doi.org/10.1080/00927872.2021.1976791}{10.1080/00927872.2021.1976791}.

\bibitem{SokolovTumanova2023SMJ}
\textit{Sokolov\hspace{.7ex}E.\hspace{.5ex}V.,\hspace{.9ex}Tumanova\hspace{.7ex}E.\hspace{.5ex}A.}\hspace{1ex}The~root-class residuality of~some generalized free products and HNN-ex\-ten\-sions, Sib.\ Math.~J.\ 64 (2) (2023) 393\nobreakdash--406. DOI:~\href{https://doi.org/10.1134/S003744662302012X}{10.1134/S003744662302012X}.

\bibitem{Gruenberg1957PLMS}
\textit{Gruenberg\hspace{.7ex}K.\hspace{.5ex}W.}\hspace{1ex}Residual properties of~infinite soluble groups, Proc.\ London Math.\ Soc.\ s3\nobreakdash-7 (1) (1957) 29\nobreakdash--62. DOI:~\href{http://doi.org/10.1112/plms/s3-7.1.29}{10.1112/plms/s3-7.1.29}.

\bibitem{Sokolov2015CA}
\textit{Sokolov\hspace{.7ex}E.\hspace{.5ex}V.}\hspace{1ex}A~characterization of~root classes of~groups, Comm.\ Algebra 43~(2) (2015) 856\nobreakdash--860. DOI:~\href{http://doi.org/10.1080/00927872.2013.851207}{10.1080/00927872.2013.851207}.

\bibitem{Sokolov2024SMJ}
\textit{Sokolov\hspace{.7ex}E.\hspace{.5ex}V.}\hspace{1ex}On~the~separability of~abelian subgroups of~the~fundamental groups of~graphs of groups.~II, Sib.\ Math.~J.\ 65 (1) (2024) 174\nobreakdash--189. DOI:~\href{http://doi.org/10.1134/S0037446624010166}{10.1134/S0037446624010166}.

\bibitem{Bardakov2004SMJ}
\textit{Bardakov\hspace{.7ex}V.\hspace{.5ex}G.}\hspace{1ex}On~D.~I.~Moldavanskii's question about $p$\nobreakdash-sep\-a\-ra\-ble subgroups of~a~free group, Sib.\ Math.~J.\ 45 (3) (2004) 416\nobreakdash--419. DOI:~\href{https://doi.org/10.1023/B:SIMJ.0000028606.51473.f7}{10.1023/B:SIMJ.0000028606.51473.f7}.

\bibitem{Hall1949TAMS}
\textit{Hall\hspace{.7ex}M.,\hspace{.5ex}Jr.}\hspace{1ex}Coset representations in~free groups, Trans.\ Am.\ Math.\ Soc.\ 67 (2) (1949) 421\nobreakdash--432. DOI:~\href{https://doi.org/10.2307/1990483}{10.2307/1990483}.

\bibitem{Stebe1968CPAM}
\textit{Stebe\hspace{.7ex}P.}\hspace{1ex}Residual finiteness of~a~class of~knot groups, Comm.\ Pure Appl.\ Math.\ 21 (6) (1968) 563\nobreakdash--583. DOI:~\href{https://doi.org/10.1002/cpa.3160210605}{10.1002/cpa.3160210605}.

\bibitem{Romanovskii1969MUSSRI}
\textit{Romanovskii\hspace{.7ex}N.\hspace{.5ex}S.}\hspace{1ex}Finite approximability of~free products with~respect to~occurrence, Math.\ USSR-Izv.\ 3 (6) (1969) 1245\nobreakdash--1249. DOI:~\href{https://doi.org/10.1070/IM1969v003n06ABEH000843}{10.1070/IM1969v003n06ABEH000843}.

\bibitem{SokolovTumanova2020IVM}
\textit{Sokolov\hspace{.7ex}E.\hspace{.5ex}V.,\hspace{.9ex}Tumanova\hspace{.7ex}E.\hspace{.5ex}A.}\hspace{1ex}On~the~root-class residuality of~certain free products of~groups with~normal amalgamated subgroups, Russ.\ Math.\ 64 (3) (2020) 43\nobreakdash--56.\newline DOI:~\href{http://doi.org/10.3103/S1066369X20030044}{10.3103/S1066369X20030044}.

\bibitem{SokolovTumanova2016SMJ}
\textit{Sokolov\hspace{.7ex}E.\hspace{.5ex}V.,\hspace{.9ex}Tumanova\hspace{.7ex}E.\hspace{.5ex}A.}\hspace{1ex}Sufficient conditions for~the~root-class residuality of~certain generalized free products, Sib.\ Math.~J.\ 57 (1) (2016) 135\nobreakdash--144. DOI:~\href{http://doi.org/10.1134/S0037446616010134}{10.1134/S0037446616010134}.

\bibitem{AzarovTieudjo2002PIvSU}
\textit{Azarov\hspace{.7ex}D.\hspace{.5ex}N.,\hspace{.9ex}Tieudjo\hspace{.7ex}D.}\hspace{1ex}On~the~root-class residuality of~a~free product of~groups with~an~amalgamated subgroup, Nauch.\ Tr.\ Ivanov.\ Gos.\ Univ.\ Math.\ 5 (2002) 6\nobreakdash--10 (in~Russian). See also: \textit{Aza\-\mbox{rov\hspace{.7ex}D.\hspace{.5ex}N.,}\hspace{.9ex}Tieudjo\hspace{.7ex}D.}\hspace{1ex}On~root-class residuality of~generalized free products, arXiv:math/0408277 [math.GR]. DOI:~\href{https://doi.org/10.48550/arXiv.math/0408277}{10.48550/arXiv.math/0408277}.

\bibitem{MagnusKarrassSolitar1974}
\textit{Magnus\hspace{.7ex}W.,\hspace{.9ex}Karrass\hspace{.7ex}A.,\hspace{.9ex}Solitar\hspace{.7ex}D.}\hspace{1ex}Combinatorial group theory (Interscience Publishers, New~York, 1966).

\bibitem{KarrassSolitar1970TAMS}
\textit{Karrass\hspace{.7ex}A.,\hspace{.9ex}Solitar\hspace{.7ex}D.}\hspace{1ex}The~subgroups of~a~free product of~two groups with~an~amalgamated subgroup, Trans.\ Am.\ Math.\ Soc.\ 150 (1) (1970) 227\nobreakdash--255. DOI:~\href{https://doi.org/10.1090/S0002-9947-1970-0260879-9}{10.1090/S0002-9947-1970-0260879-9}.{\parfillskip=0pt\par}





\end{thebibliography}
\end{document}